\documentclass{amsart}
\usepackage{amssymb}
\usepackage{appendix}
\usepackage[symbol]{footmisc}
\usepackage{graphicx}
 \usepackage[all]{xy}
\usepackage{geometry}
\usepackage{fancyhdr}
\usepackage{xcolor}
\usepackage{boxhandler}
\usepackage{caption}
\usepackage{lipsum}
\newcommand{\vertiii}[1]{{\left\vert\kern-0.25ex\left\vert\kern-0.25ex\left\vert #1 
    \right\vert\kern-0.25ex\right\vert\kern-0.25ex\right\vert}}

\newtheorem*{theorem*}{Main Theorem}
\newtheorem{theorem}{Theorem}

\newtheorem{lemma}{Lemma}
\newtheorem{prop}[theorem]{Proposition}

\newtheorem{coro}[theorem]{Corollary}
\newtheorem*{fact}{Fact}
\newtheorem{rem}[theorem]{Remark}
\theoremstyle{definition}

\newtheorem*{ques}{Question}

\newcommand{\gr}{\mathop{\mathrm{gr}}}
\newcommand{\mdim}{\mathop{\mathrm{mdim}}}

\newcommand{\cv}{\mathop{\mathrm{cv}}}
\newcommand{\ex}{\mathop{\mathrm{ex}}}
\newcommand{\diam}{\mathop{\mathrm{diam}}}
\newcommand{\Int}{\mathop{\mathrm{Int}}}

\theoremstyle{remark}

\numberwithin{equation}{section}

\begin{document}

\title{Rescaled Entropy of   cellular automata}

\author{David Burguet}
\address{Sorbonne Universite, LPSM, 75005 Paris, France}
  \email{david.burguet@upmc.fr}  

\date{June 2018}



\begin{abstract}For a $d$-dimensional cellular automaton with $d\geq 1$ we introduce a rescaled entropy which estimates the growth rate of the entropy at small scales by generalizing previous approaches \cite{Bla,Lak}. We also define a notion of Lyapunov exponent and proves a Ruelle inequality as already established for $d=1$ in \cite{Ti,Sh}. 
Finally we generalize the entropy formula for $1$-dimensional permutative cellular automata \cite{Wa} to the rescaled entropy in higher dimensions. This last result extends recent works \cite{Tsu} of Shinoda and Tsukamoto dealing with the metric mean dimensions of  two-dimensional symbolic dynamics. \end{abstract}

\maketitle
\section{Introduction}
In this paper we estimate the dynamical complexity of multidimensional cellular automata. In the following the main results will be stated in a more general setting, but let us focus  in this introduction on the following algebraic cellular automaton on $\left(\mathbb F_p\right)^{\mathbb Z^d}$ with $p$ prime  given for some finite family $(a_i)_{i\in I}$ in $\mathbb F_p^*$ by 
$$\forall (x_j)_j\in \left(\mathbb F_p\right)^{\mathbb Z^d}, \  f((x_j)_j)=\left(\sum_{i\in I}a_ix_{i+j}\right)_j.$$

Let $I'=I\cup\{0\}$. For $d=1$ the topological entropy of $f$ is finite and equal to $\diam (I')\log p$ where $\diam(I')$ denotes  the diameter of $I'$ for the usual distance on $\mathbb R$ \cite{Wa}. However in higher dimensions the topological entropy of $f$ is always infinite unless $f$ is the identity map \cite{MW, Lak1}. Moreover the topological entropy of the $\mathbb Z^{d+1}$-action given by $f$ and the shift vanishes. In this paper we investigate the growth rate of 
$\left(h_{top}(f, \mathsf P_{J_n})\right)_n$  for nondecreasing sequences $(J_n)$ of convex subsets of $\mathbb R^d$ where $ \left(\mathsf P_{J_n}\right)_n$ denotes the clopen partitions into $\underline{J_n}$-coordinates with $\underline{J_n}:=J_n\cap \mathbb Z^d$. This sequence appears to increase   as  the perimeter $p(J_n)$ of $J_n$. We define the rescaled entropy $h_{top}^d(f)$ of $f$ as $\limsup_{J_n}\frac{h_{top}(f, \mathsf P_{J_n})}{p(J_n)}$.   In \cite{Lak}  another renormalization is used, whereas in \cite{Bla} the authors only investigate the case of squares $J_n=[-n,n]^2, \ n\in \mathbb N$. For $d=1$ we get $h_{top}^1(f)=\frac{h_{top}(f)}{2}$.  We generalize the entropy formula  for algebraic cellular automata as follows :
\begin{theorem}\label{pm}
Let $f$ be an algebraic cellular automaton on $\left(\mathbb F_p\right)^{\mathbb Z^d}$  as above, then 
$$h_{top}^d(f)=R_{I'}\log p,$$
where $R_{I'}$ denotes the radius of the smallest bounding sphere containing $I'$. 
\end{theorem}

In fact we establish such a formula for any permutative cellular automaton (see Section \ref{las}).  In \cite{Tsu} the authors compute, inter alia,  the metric mean dimension of the horizontal shift in $\mathbb Z^2$ for some standard distances. These dimensions may be interpreted as the rescaled entropy with respect to some particular sequence of convex sets $(J_n)_n$. In particular we extend these results  in higher dimensions for general permutative cellular automata.

We also consider a measure theoretical analogous quantity of the rescaled entropy. In dimension one, a notion of Lyapunov exponent has been defined in \cite{Sh}. Then 
Tisseur \cite{Ti}  proved in this case a Ruelle inequality  relating this exponent with the Kolmogorov-Sinai entropy. In this paper  we also introduce a notion of Lyapunov exponent in higher dimensions, which bounds from above the rescaled entropy of measures.

The paper is organized as follows. In Section 2 we state some  measure geometrical properties of convex sets in $\mathbb R^d$. We recall the dynamical background of cellular automata in Section \ref{deff} and we introduce then a Lyapunov exponent for multidimensional cellular automata. In Section \ref{entr} we define and study the topological and measure theoretical rescaled entropy. We prove the Ruelle type inequality in Section 6. The last section is devoted to the proof of the entropy formula for  permutative cellular 
automata.

\section{Background on convex geometry}\label{conv}

\subsection{Convex bodies, domains and polytopes}

For a fixed positive integer $d$ we  endow the vector space  $\mathbb R^d$ with its usual Euclidean structure.  The associated scalar product is simply denoted by $\cdot$ and we let $\mathbb S^d$ be the unit sphere.  For a subset $F$ of $\mathbb R^d$ we let  $\overline F$, $\Int(F)$ and $ \partial F$ be respectively its closure,  interior set  and  boundary. We denote by $\underline{F}$ the set of integer points in $F$, i.e. $\underline{F}=F\cap \mathbb Z^d$.   We also denote by $V(F)$  the $d$-Lebesgue measure of $F$ (also called the volume of $F$) when the set $F$ is Borel.  

The extremal set of a convex set $J$ is denoted by $\ex(J)$ and the convex hull of $F\subset \mathbb R^d$ by $\cv(F)$. A convex body is a compact convex set of $\mathbb R^d$.  A convex body containing the origin $0\in \mathbb R^d$ in its interior set  is said to be  a \textbf{convex domain}. The set of convex bodies  endowed with the Hausdorff topology is a locally compact metrizable space. In the following we denote by $\mathcal D$, resp. $\mathcal D^1$, the set of convex domains, resp. with unit perimeter, endowed with the Hausdorff topology.    A \textbf{convex  polytope} (resp. $k$-polytope with $k\leq d$) in $\mathbb R^d$ is a convex body given by the convex hull of a finite set (resp. with topological dimension equal to $k$). When this  finite set lies inside the lattice $\mathbb Z^d$, the convex polytope is said \textbf{integral}. We let $\mathcal F(P)$ be  the set of faces of a convex polytope $P$. For a convex body $J$ we denote by $\mathsf J$  the integral polytope given by the convex hull of integer points in $J$, i.e.  $\mathsf J= \cv(\underline{J})$.

 A  convex domain $J$  has Lipshitz boundary and  finite perimeter $p(J)$. For convex domains the perimeter in the distributional sense of   De Giorgi coincides with the $(d-1)$-Hausdorff measure $\mathcal{H} _{d-1}$ of the boundary. For $J\in \mathcal D$ we let $\partial ' J$ be the subset of points $x\in \partial J$, where the tangent space $T_xJ$ is well defined. The set  $\partial 'J$  has full $\mathcal{H} _{d-1}$-measure  in $\partial J$.  We let $N^J(x)\in \mathbb S^d$ be the unit $J$-external normal vector at  $x\in \partial'J$. For any $x\in \partial 'J$ we let $T_x^+J$ (resp. $T_x^-J$) be the open external (resp. closed internal) semi-space with boundary $T_xJ$. With these notations we have 
$J=\bigcap_{x\in \partial 'J}T_x^-J$. For $\epsilon\in \mathbb R$ we denote by  $T_x^{\pm}J(\epsilon)$ the semi-planes
$T_x^\pm J(\epsilon)= T_x^\pm J+ \epsilon N^J(x)$. When $J$ is a convex polytope and $F\in \mathcal F (J)$, we write  $T_F$  to denote the tangent affine space supporting $F$,  $T_F^{\pm}$ for the associated   semi-spaces and $N^F$ for  the unit external normal to $F$.  

The \textbf{support function} of a convex body $I$ is the real  continuous function $h_I$ on $\mathbb S^d$ : 
$$\forall x\in \mathbb S^d, \ h_I(x)=\max_{u\in I}u\cdot x.$$ 
The support function completely characterizes the convex body $I$.  The \textbf{area measure} $\sigma_J$ of a convex domain $J$ is the Borel measure on $\mathbb S^d$  given by $N^J_*\mathcal{H} _{d-1}$ :
$$\forall B \text{ Borel of }\mathbb S^d, \ \sigma_J(B)=\mathcal{H} _{d-1}\left((N^J)^{-1}B\right).$$ 
If a sequence $(J_n)_n$ in $\mathcal D$ is converging to $J_\infty\in \mathcal D$ (for the Hausdorff topology), then $\sigma_{J_n}$ is converging weakly to $\sigma_{J_\infty}$, in particular the perimeter of $J_n$ goes to the perimeter of $J_\infty$ (see Proposition 10.2 in \cite{Gru}).

\subsection{Convex exhaustions}

 We consider    sequences $\mathcal J=(J_n)_{n\in \mathbb N}$ of convex domains    with $p( J_n)\xrightarrow{n} +\infty$, such that the sets 
$\tilde{J_n}=p\left( J_n\right)^{-\frac{1}{d-1}}J_n \in \mathcal D^1$ are  converging  to a limit  $J_\infty\in \mathcal D$ in the Hausdorff topology.  In particular $\bigcup_nJ_n=\mathbb R^d$.  Moreover  the limit $ J_\infty$ has unit perimeter. The sequences  $\mathcal J= (J_n)_n$  satisfying the  above properties are said to be  \textbf{convex exhaustions}. 
For $O\in \mathcal D^1$    we denote by  $\mathcal{E}(O)$  the set of convex exhaustions $\mathcal J=(J_n)_n$ with $J_\infty=O$. Moreover for $O\in \mathcal D$ we let  $\mathcal J_O\in \mathcal E\left(p(O)^{-\frac{1}{d-1}}O\right)$ be the convex exhaustion given by $\mathcal J_O:=\left(nO\right)_n$. A convex exhaustion $(J_n)_n$ is said integral when $J_n$ is an integral polytope for all $n$.

The inner radius $r(E)$ of a subset $E$ of $\mathbb R^d$ is the largest $a\geq 0$ such that $E$ contains a Euclidean ball of radius $a$. 
For two subsets $E$ and $F$ of $\mathbb R^d$ we denote $E\Delta F $ the symmetric difference of $E$ and $F$ given by $E\Delta F :=\left(E\setminus F\right)\cup \left(F\setminus E\right)$.

\begin{lemma}\label{nul}
Let $O\in \mathcal D$ and  $\mathcal J =(J_n)_n\in \mathcal E(O)$. Then any sequence of convex bodies $\mathcal K=(K_n)_n$ with $r\left(K_n\Delta J_n\right)=o\left(p(J_n) \right)$ belongs to $\mathcal E(O)$ and $p(K_n)\sim^n p(J_n)$. 
\end{lemma}

\begin{proof}
We claim that $p\left( J_n\right)^{-\frac{1}{d-1}}K_n$ is converging to $J_\infty$ in the Hausdorff topology. Then by taking the perimeter in this limit we get $\lim_n\frac{p(K_n)}{p(J_n)}=p(J_\infty)=1$ and therefore $\tilde{K_n}=p\left( K_n\right)^{-\frac{1}{d-1}}K_n$ also goes to $J_\infty=O$.
Let us prove now the claim. Fix a Euclidean ball $B$ with $J_\infty \subset \Int B$. It is enough to show that $p\left( J_n\right)^{-\frac{1}{d-1}}K_n\cap B$ is converging to $J_\infty$. Indeed as $K_n$ is convex, this will imply that $p\left( J_n\right)^{-\frac{1}{d-1}}K_n\subset B$ lies in $B$ for $n$ large enough (if not $p\left( J_n\right)^{-\frac{1}{d-1}}K_n\cap \partial B$ is non empty for infinitely many $n$ and therefore we should have $J_\infty\cap \partial B\neq \emptyset$). By extracting a subsequence we may assume  $p\left( J_n\right)^{-\frac{1}{d-1}}K_n\cap B$ is converging to a convex body  $K_\infty$ and we need to prove $K_\infty=J_\infty$. We argue by contradiction. As $J_\infty$ is a convex domain, we have either $\Int(J_\infty)\setminus K_\infty\neq \emptyset $ or $\Int(K_\infty)\setminus J_\infty\neq\emptyset$. But  for $x$ in one of these sets, there is $s>0$  such that the balls $p(J_n)B(x,s)$ are contained in $K_n\Delta J_n$, therefore $r\left(K_n\Delta J_n\right)\geq sp(J_n)$,  for $n$ large enough.
\end{proof}

\begin{rem}\label{integr}
If $\sharp \underline{K_n\Delta J_n}= o\left( p(J_n)^{\frac{d}{d-1}} \right)$ then the condition on the inner radius in Lemma \ref{nul} holds and therefore $\mathcal K$ belongs to  $\mathcal E(O)$.  In particular $(\mathsf J_n)_n$ is a convex exhaustion in $\mathcal E(O)$. 
\end{rem}

\subsection{Internal and external morphological boundary}

We recall some terminology of mathematical morphology used in image processing. For two subsets $I$ and $J$ of $\mathbb R^d$, the \textbf{dilation} (also known as the Minkowski sum) $J\oplus I $ and the \textbf{erosion} $J\ominus I$ of $J$ by $I$ are defined as follows
\begin{align*}
J\oplus I&=\{i+j \ | \ i\in I \text{ and  }j\in J\},\\
J\ominus I&=\{j\in \mathbb R^d \ | \ \forall i\in I, \, i+j\in J \}.
\end{align*}

When the origin $0$ belongs to $I$ then we have  $J\subset J\oplus I$ and $J\ominus I\subset J$. 
  When $J$ is a convex body  then $J\ominus I$ is a convex body. Assume now that  $I$ is also a convex body. The dilation  $J\oplus I$ is then also a convex body with $\ex(J
\oplus I)\subset \ex(I)\oplus \ex(J)$. In particular, when $I$ and $J$ are moreover convex polytopes, then so is $J\oplus I$. We have  $J\ominus I=\bigcap_{x\in \partial 'J}T_x^-J\left(h_I(-N^J(x))\right)$  (also $J\oplus I\subset\bigcap_{x\in \partial 'J}T^-_xJ\left(h_I(N^J(x))\right)$, but this last  inclusion may be strict). When $J$ is a convex polytope, the  above intersection is  finite, thus $J\ominus I$  is also a convex polytope.  The convex bodies given by   the erosion $J\ominus I$ and the dilation $J\oplus I$ are also   known as the inner and outer parallel bodies of $J$ relative to $I$.  We recall that $h_{J\oplus I}=h_J+h_I$. In particular when $I=\{i\}$ is a singleton, we get $h_{J+i}(x)=h_J(x)+i\cdot x$ for all $x\in \mathbb S^d$. In general we only have $h_{J\ominus I}\leq h_J-h_I$.

The \textbf{internal and external (morphological) boundaries} of $J$ relative to $I$ denoted respectively  by $\partial^-_I J$ and $\partial^+_I J$ are given by
 \begin{align*}
\partial_I^+J &=\left(I\oplus J\right)\setminus J,\\
\partial^-_IJ&=J\setminus(J\ominus I).
\end{align*} 
Clearly we have $\partial_{I}^\pm J =\partial_{I'}^\pm J$ with $I'=I\cup\{0\}$. When $J$ is a convex domain then we have $\partial_{I}^- J =\partial_{\cv(I)}^- J$ and $\partial_{I}^+ J \subset \partial_{\cv(I)}^+ J$. In the following the set $I$ will be fixed 
so that  we omit the index $I$ in the above definitions when there is no confusion.

Finally we observe that  $r\left( J_n\Delta (J_n\oplus I)\right), \, r\left( J_n\Delta (J_n \ominus I)\right)\leq \diam(I')$. Therefore  it follows from Lemma \ref{nul}, that if $(J_n)_n$ is a convex 
exhaustion and $I$ a convex body   then $(J_n\ominus I)_n$ and $(J_n\oplus I)_n$ define convex exhaustions with the same limit as $(J_n)_n$.

\section{Counting integer points in morphological boundary of large convex sets}

For a large convex domain $J$ and a fixed  integral polytope $I$ we estimate the cardinality  of the integer points in the morphological boundaries of $J$ relative to $I$.
We first compare the cardinality of integer points in the internal and external boundaries of $J$ and  $\mathsf J$. Recall that $\underline{F}$ denotes the set of integer points in a subset set $F$ of $\mathbb R^d$ and $\mathsf J=\cv(\underline{J})$. 

\begin{lemma}\label{jama} With the above notations we have 
\[ \underline{\partial^{-}\mathsf J} = \underline{\partial^{-}J} \]
and 
 \[ \underline{\partial^+\mathsf J}\subset \underline{ \partial^+ J}.\]
\end{lemma}

In general the last inclusion is strict.

\begin{proof}
For any convex domain $J$, a  point $u$ of $J$ belongs to $\partial^-J$  if and only if there is $v$ in $\ex(I)$ such that $u+v$ does not lie in $J$. As $ J\cap \mathbb Z^d=\mathsf  J\cap \mathbb Z^d$ and $\ex(I)\subset \mathbb Z^d$, we get  $ \underline{\partial^{-}J }=  \underline{\partial^{-}\mathsf J}$.  
Similarly if  a point $u\in \partial^+\mathsf J$ is an integer, then $u\in J\oplus I$ but $u\notin J$. Therefore we get $\underline{\partial^+\mathsf J}\subset \underline{\partial^+ J}$. 
\end{proof}

\begin{lemma}\label{prof}Let $J$ be a convex polytope. 
$$\sharp \underline{\partial^{-}J }\leq \sharp\underline{\partial^{+}J }.$$
\end{lemma}

\begin{proof}
We have $\partial^-J\subset \bigcup_{F\in \mathcal F(J)}T_F^+J(-h_I(N^F))$. For  $F\in \mathcal F(J)$ there exists $u^F\in \ex(I)$ with $h_I(N^F)=u^F\cdot N^F$. 
Let $F_1,\cdots F_N$ be an enumeration of $\mathcal F(J)$. Let $\phi:\partial^-J\rightarrow \partial^+J$ be the function  defined by 
$\phi(x)=x+u^{F_1}$ for $x\in S_1:= \partial^-J\cap T_{F_1}^+J(-h_I(N^{F_1}))$ and  $\phi(x)=x+u^{F_l}$ for $x\in S_l:= \partial^-J\cap  T_{F_l}^+J(-h_I(N^{F_l}))\setminus \bigcup_{k<l} T_{F_k}^+J(-h_I(N^{F_k}))$  by induction on $l$.

This map is injective : indeed if $\phi(x)=\phi(y)$ either $x$ and $y$ lie in the same $S_l$ and then $\phi(x)=x+u^{F_l}=y+u^{F_l}=\phi(y)$ clearly implies $x=y$ or $x\in 
S_k, y\in S_l$ with $k\neq l$. We may assume $k<l$ without loss of generality. Then $y+u^{F_l}\in T_{F_k}^-J$ whereas $x+u^{F_k}\in T^+_{F_k}J$ and we get 
thus a contradiction. Finally the map $\phi$ preserves the integer points  since we have $\ex(I)\subset \mathbb Z^d$.
\end{proof}

\subsection{First relative quermass integral}

Let $O$ be a convex domain and let  $I$ be a convex body. For $\rho \in \mathbb R$ we let 
$$ O_\rho =\left\{
\begin{array}{rl}
  O\oplus \rho I  \text{ when }\rho\geq 0, \\
O\ominus \rho I \text{ when } \rho<0.
\end{array}
\right.$$

\begin{prop}\label{voll}
$$\lim_{\rho\rightarrow 0}\frac{V(O_\rho)-V(O)}{\rho}=\int_{\mathbb S^d}h_I \,d\sigma_O .$$
\end{prop}
For $\rho >0$ the formula follows from Minkowski's formula on mixed volume (see Theorem 6.5 and Corollary 10.1 in \cite{Gru}). For $\rho <0$ 
we refer to \cite{Mat} (see also Lemma 2 in \cite{Cha} for the $2$-dimensional case). \\

The  quantity  $d\int_{\mathbb S^d}h_I \,d\sigma_O $ is known as the \textbf{first $I$-relative quermass integral} of $J$. In the following we denote by $V_I(O)$ the integral $\int_{\mathbb S^d}h_I \,d\sigma_O $.   For convex bodies $I\subset H$ and $k\in \mathbb N$, we have $V_{I}(O)\leq V_{H}(O)$ and $V_{kI}(O)=kV_I(O)$  for any convex domain $O$.  The support  function $h_I$ being continuous, the first $I$-relative quermass integral of $O$ is continuous with respect to the Hausdorff topology, i.e. if $(O_n)_n$ is a sequence of convex domains converging to a convex domain $O_\infty$ in the Hausdorff topology, 
then we have $$V_I(O_n)\xrightarrow{n\rightarrow +\infty}V_I(O_\infty).$$

We deduce now from Proposition \ref{voll} an estimate on the volume of the morphological boundary for large convex sets.

\begin{coro}\label{reff}
Let $I$ be a convex body containing $0$ and let $O\in \mathcal D$. Then 
$$V\left(\partial^{\pm}_InO\right)\sim n^{d-1}\int_{\mathbb S^d}h_I \,d\sigma_O .$$
\end{coro}

\begin{proof}
We only consider the case of the external boundary as one may argue similarly for the internal boundary.  For all $n$ we have 
\begin{align*}
V\left(\partial^{+}_InO\right)&=V\left( nO\oplus I \right)-V\left(nO\right),\\
&= n^{d}\left(V(O\oplus n^{-1}I)-V(O)\right)
\end{align*}
According to Proposition \ref{voll} we conclude that 
\begin{align*}
V\left(\partial^{+}_InO\right)&\sim n^{d-1}\int_{\mathbb S^d}h_I \,d\sigma_O.
\end{align*}

\end{proof}

\subsection{Counting integer points in large convex sets}

Since Gauss circle problem counting lattice points in convex sets has been extensively investigated. Let $\mathsf C=[0,1]^d$. Clearly  for any Borel subset $K$ of $\mathbb R^d$  we have always
\begin{equation}\label{fac}\sharp \underline{K}
\leq V(K\oplus \mathsf C).\end{equation}  

In the other hand, Bokowski, Hadwiger and  Wills have proved the following general (sharp) inequality for any convex domain $O$ \cite{BHW} :
\begin{equation}\label{bol}V(O)-\frac{p(O)}{2}\leq \sharp \underline{O}.
\end{equation}
There exist precise asymptotic estimates of $\sharp \underline{x O}$ for large $x>0$ for convex  smooth domains $O$ having positive curvature, in particular we have in this case $\sharp \underline{x O}=V(xO)+o(x^{d-1})$ \cite{Hla}. 

\subsection{First rough estimate for $\sharp\partial_I^\pm nO\cap \mathbb Z^d$ with $O\in \mathcal D$}
For a real sequence $(a_n)_n$ and two numbers $l$ and $C>0$ we write $a_n\sim^C l$ when the accumulation points of $(a_n)_n$ lie in 
$[l-C,l+C]$.
\begin{lemma}\label{grg}There exists a constant $C$ depending only on $d$ such that we have for any 
 convex domain $O\in \mathcal D$ and any convex body $I$ of $\mathbb R^d$ with $0\in I$ :
 $$\frac{\sharp\underline{ \partial_I^\pm nO }}{p(nO)}\sim^{C}  \frac{V_I\left(O\right)}{p(O)}.$$
\end{lemma}

\begin{proof}
We only argue for $\partial^+_IO$, the other case being similar. We have $\sharp \underline{\partial_I^+ nO}=\sharp \underline{nO\oplus I}-\sharp \underline{nO}$, and then by combining Equation (\ref{fac}) and (\ref{bol}) we get :
\begin{eqnarray*}
V(nO\oplus I)-\frac{p(nO\oplus I)}{2}-V(nO+\mathsf C)&\leq \sharp \underline{\partial_I^\pm nO}&\leq V(nO\oplus I\oplus \mathsf C)-V(nO)+\frac{p(nO)}{2},\\
\end{eqnarray*}

After dividing by $n^{d-1}$, the right (resp. left) hand side term is going to $\int_{\mathbb S^d}(h_I-h_{\mathsf C}-1/2) \,d\sigma_O$ (resp. $\int_{\mathbb S^d}(h_{I}+h_{\mathsf C}+1/2)\,d\sigma_O$ )  according to Corollary \ref{reff}.
\end{proof}

\subsection{Upperbound of $\partial^-J_n$ for general convex exhaustions}

For a subset $E$ of $\mathbb R^d$ and for $r>0$ we let $E(r):=\{x\in E, d(x,\partial E)\leq r\}$ with $d$ being the Euclidean distance.
With the previous notations we may also write $E(r)=\partial^-_{B_r}E$ where $B_r$ denotes the Euclidean  ball centered at $0$ with radius $r$.

\begin{lemma}\label{zut}
For any  convex body $J$  in $\mathbb R^d$, we have 
$$V\left(J(r)\right)\leq  r p(J).$$
\end{lemma}
\begin{proof}
We first assume  that $J$ is a convex polytope. Let $x\in J(r)$. There is $F\in \mathcal F(\mathsf J)$ with $\|x-x_F\|\leq d(x,F)=d(x, \partial J)\leq r$, where $x_F$ denotes the orthogonal projection of $x$ onto $T_F$. Observe that $x_F$ belongs to $F$ : if not the segment line $[x,x_F]$ would have a non empty intersection with $\partial J$ and the intersection point $y\in \partial J$
 would satisfy $\|x-y\|< \|x-x_F\|\leq d(x, \partial J)$. Therefore $J(r)\subset \bigcup_{F\in \mathcal F(J)}R_F(r)$ with $R_F(r):=\{x-tN^F(x), \ x\in F \text{ and } t\in [0,r]\}$. Finally we get 
 \begin{align*}
 V\left(J(r) \right)&\leq \sum_{F\in \mathcal F(J)}V\left(R_F(r)\right),\\
 &\leq rp(J).
 \end{align*}
For a general convex body, there is a nondecreasing sequence $(J_p)_p$ of convex polytopes contained in $J$ converging to $J$ in the Hausdorff topology. Then the characteristic function of $J_p(r)$ is converging pointwisely to the characteristic function of $J(r)$, in particular $V\left(J_p(r)\right)\xrightarrow{p}V\left(J(r)\right)$. Moreover $p(J_p)$ goes to $p(J)$, so that the desired inequality is obtained by taking the limit in the inequalities for the convex polytopes $J_p$. 
\end{proof}

\begin{prop}\label{fin}
For any convex exhaustion $(J_n)_n$ in $\mathbb R^d$, we have 
$$\limsup_n\frac{\sharp \underline{\partial_I ^- J_n}}{p\left(J_n\right)}\leq \diam(I')+\sqrt d.$$
\end{prop}

\begin{proof}
 As already observed, we have $\sharp \underline{\partial^-J_n}\leq V(\partial^-J_n\oplus\mathsf C)$ with $\mathsf C=[0,1]^d$. 
 Let $(J'_n)_n$ be the sequence given by $J'_n=J_n\oplus C$ for all $n$. By Lemma \ref{nul} this sequence is a convex exhaustion with $p(J'_n)
 \sim^n p(J_n)$.  Moreover $ \partial^-J_n\oplus C$  is contained in $J'_n\left( c \right)$  with $c=\diam(I')+\diam(\mathsf C)$. Therefore we conclude according to 
 Lemma \ref{zut}  :
 \begin{align*}
 \sharp \underline{\partial^-J_n}&\leq V\left(J'_n(c)\right),\\
 &\leq c p(J'_n),\\
 &\lesssim^n cp(J_n).
 \end{align*}
 
\end{proof}

\subsection{Fine estimate of $\sharp\underline{\partial_I^\pm J_n}$ for general convex exhaustions $(J_n)_n$ in dimension $2$}
We compare directly the cardinality of lattice points in the morphological boundary with the first $I$-relative quermass integral  of $J_\infty$ for two-dimensional convex exhaustion. This result will not be used directly in the next sections but is potentially of independent interest.

\begin{prop}\label{fdf}
For any   convex exhaustion $( J_n)_n$ in $\mathbb R^2$, we have 
 $$\lim_n\frac{\sharp \underline{\partial_I ^{-} J_n} }{p(J_n)}= V_I(J_\infty).$$
\end{prop}

 By  Remark \ref{integr} and  Lemma \ref{jama} we only need to consider integral convex exhaustions. In fact  in this case we  also show the corresponding statement for the external morphological boundary.

\begin{prop}\label{fdff}
For any integral  convex exhaustion $(\mathsf J_n)_n$ in $\mathbb R^2$, we have 
 $$\lim_n\frac{\sharp \underline{\partial_I ^{\pm}\mathsf  J_n} }{p(\mathsf J_n)}= V_I(\mathsf J_\infty).$$
\end{prop}

The rest of  this subsection is devoted to the proof of Proposition \ref{fdff}. We start by giving some preliminary lemmas.\\

We denote by $\angle P$ the minimum of the interior angles at the vertices of  a convex polygon $P\subset \mathbb R^2$. 

\begin{lemma}\label{anglee}
For any integral  convex exhaustion $(\mathsf J_n)_n$ in $\mathbb R^2$, we have 
$$\liminf_n \angle \mathsf  J_n >0.$$
\end{lemma}

\begin{proof}
We have $\angle \tilde{\mathsf J_n}=\angle \mathsf J_n$. Moreover the minimal angle is lower semi-continuous for the Hausdorff topology, therefore $\liminf_n\angle \tilde{\mathsf J_n}\geq \angle\mathsf  J_\infty$. Since $\mathsf J_\infty$ has non-empty interior, we have $\angle \mathsf  J_\infty>0$. 
\end{proof}

\begin{lemma}\label{pepe}
For any integral  convex exhaustion $(\mathsf J_n)_n$ in $\mathbb R^2$, we have $$\sharp \mathcal F(\mathsf J_n)=o\left(p(\mathsf J_n)\right).$$
\end{lemma}

\begin{proof}
Two  integral polytopes are said equivalent when there is a translation (necessarily by an integer) mapping one to the other. 
For any $L$ the number $a_L$ of equivalence classes of integral $1$-polytopes with $1$-Hausdorff measure less than $L$ is finite (these polytopes are just line segments with integral endpoints and their  $1$-Hausdorff measure is just equal to their length). Moreover for a  integral convex polytope  there are at most two faces in the same class. Therefore 
\begin{align*}
\sharp \mathcal F(\mathsf J_n)&\leq 2a_L+\sharp\{F\in \mathcal  F(\mathsf J_n), \ \mathcal{H} _{1}(F)\geq L \},\\
&\leq 2a_L+\frac{p(\mathsf J_n)}{L}.
\end{align*}
This inequality holds for all $n$ and  $p(\mathsf J_n)$ goes to infinity with $n$ so that we conclude $\sharp \mathcal F(\mathsf J_n)=o\left(p(\mathsf J_n)\right)$ as $L$ was arbitrarily fixed. 
\end{proof}

 Given two distinct points $A,B$ in $\mathbb R^2$ and $h\neq 0$, the rectangle $R_{AB}(h)$ of basis $AB$ and height $h>0$ (resp. $h<0$) is the semi-open rectangle $[AB[\times [A,D[$ oriented as  $ABCD$ (resp. $ADCB$)\footnote{We  denote a convex polytope with its vertices by respecting the usual orientation of the plane.} with $|AD|=|h|$. This rectangle is said integral when $A,B$ belong to $ \mathbb Z^2$ and the line $(CD)$ has a non-empty intersection with $\mathbb Z^2$. 


\begin{lemma}\label{ggo} For any integral rectangle $R$, 
$$\sharp \underline{R} =V(R).$$
\end{lemma}
\begin{proof}
After a translation by an integer we may assume  that the origin is the  vertex $A$  of  the integral  rectangle $R=R_{AB}(h)$. Let $(p',q')$ be an integer on the line segment $[A,B]$ with $p',q'$  relatively prime. By Bezout theorem there is $(u,v)\in \mathbb Z^2$
with $up+vq=1$. Therefore there is a matrix $M\in SL2(\mathbb Z)$ with $M(p,q)=(k,0)$. As the transformation $M$ preserve both the volume and the integer points it is 
enough to consider the semi-open  parallelogram $M(R)$. But there is a piecewise integral translation, which maps $M(R)$ to a semi-open integral rectangle with basis  $M([A,B[)\subset \mathbb R\times \{0\}$. For such a rectangle the area is obviously equal to the cardinality of its integer points.
\end{proof}

For  $A,B\in \mathbb R^2$ and $\epsilon<\frac{|AB|}{2} $ we let $A^\epsilon $ and  $B^\epsilon$ be the points in  the line $(AB)$ with Euclidean distance $|\epsilon|$ to $A$ and $B$ respectively, which lie inside $[A,B]$ if  $\epsilon>0$  and outside ifnot. As the symmetric difference of $R_{A^\epsilon B^\epsilon}(h)$ and $R_{AB}(h)$ is given by the union of two rectangles with sides of length $|\epsilon|$ and $|h|$ we have for some constant $C=C(|\epsilon|,|h|)$ 
\begin{equation}\label{bb}\left|\sharp \underline{R_{A^\epsilon B^\epsilon}(h)}-\sharp \underline{R_{AB}(h)}\right|\leq C.
\end{equation}
This estimate still holds true for $\epsilon\geq|AB|/2 $ when choosing the convention $R_{A^\epsilon B^\epsilon}(h)=\emptyset$ for such $\epsilon$.

\begin{fact}\label{hg}
For any convex body $I$ and for any $a>0$, there exists  $\epsilon^+=\epsilon^+(I)>0$ and $\epsilon^-=\epsilon^-(I,a)>0$  such that any convex polytope $J=A_1\cdots A_n$ with $\angle J\geq a$ satisfies
$$\partial^+J\subset \bigcup_{l<n}R_{A_l^{\epsilon^+} A_{l+1}^{-\epsilon^+}}\left(-h_I(N^{A_lA_{l+1}})\right)$$ 
and $$\partial^-J\supset \bigcup_{l<n }R_{A_l^{\epsilon^-} A_{l+1}^{-\epsilon^-}}\left(h_I(N^{A_lA_{l+1}})\right).$$
\end{fact}

\begin{figure}[!ht]
\includegraphics[scale=0.4]{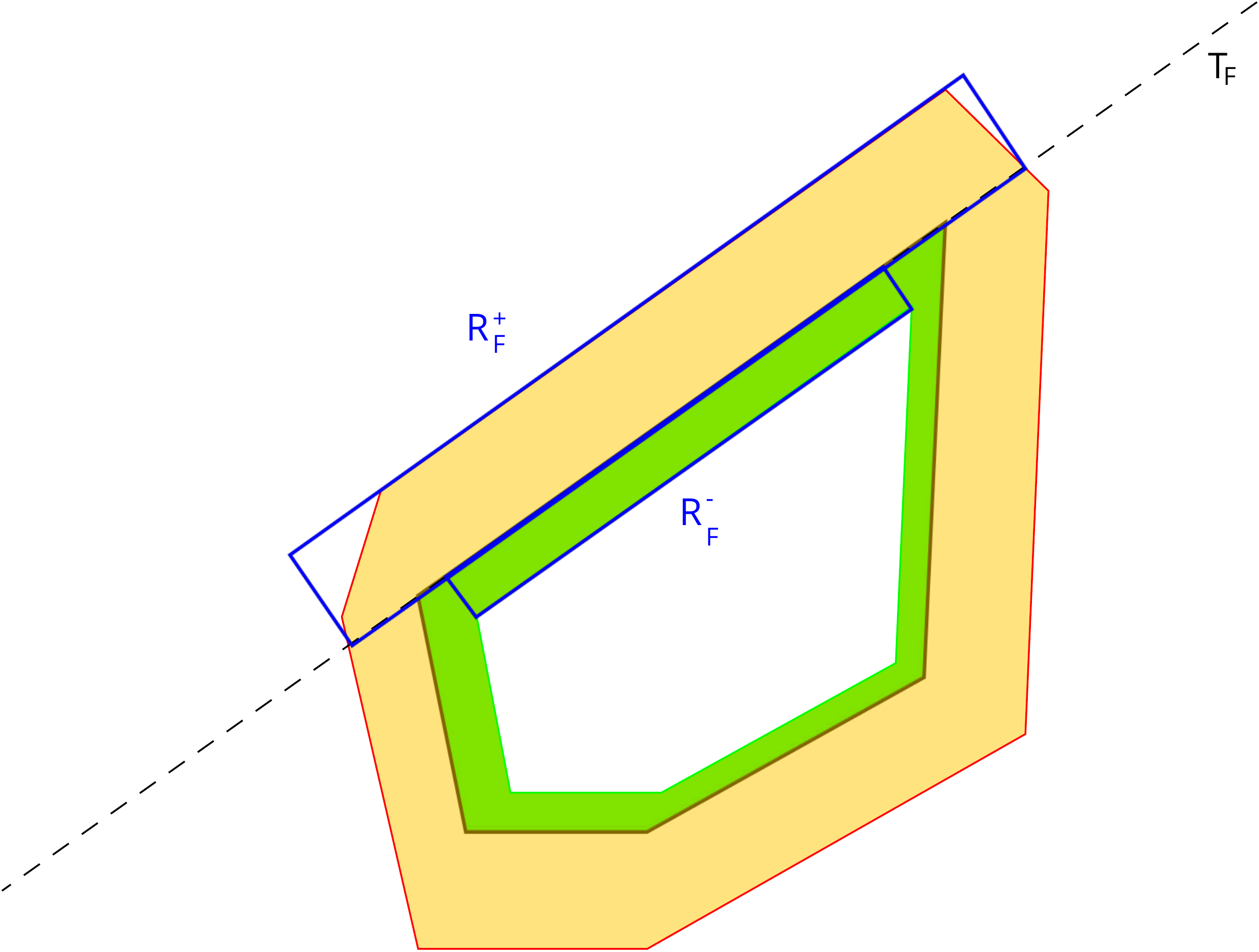}
\centering
\caption{\label{bl}\textbf{The external  and  internal   rectangles  associated to a face $F$ of a polygon.}  The external and internal  morphological boundaries are respectively represented  by the areas in yellow and green. The rectangles,  $R_F^+$ and  $R_F^-$, given by Fact \ref{hg}  are drawn in blue.}
\label{figg}
\end{figure}

This fact is illustrated on Figure \ref{figg} and its easy proof is left to the reader. We are now in a position to prove Proposition \ref{fdff}.

\begin{proof}[Proof of Proposition \ref{fdff}]
From the above fact and (\ref{bb})  there is $\epsilon=\epsilon^-(I,\angle J)>0$ and $C=C(I,\angle J)>0$ such that for any convex polytope $J=A_1\cdots A_n$
\begin{align*}\sharp \underline{\partial^-J}&\geq  \sum_{l<n} \sharp \underline{R_{A_l(\epsilon)A_{l+1}(\epsilon)}\left(-h_I(N^{A_lA_{l+1}})\right)},\\
& \geq  \sum_{F\in \mathcal F(J)} \left[\sharp \underline{R_{F} \left(-h_I(N^{F})\right)}-C\right].
\end{align*}
Then when $J$ is an integral convex polytope we get by Lemma \ref{ggo} :
 \begin{align*}\sharp \underline{\partial^-J}&\geq  -C\sharp \mathcal F(J)+ \sum_{F\in \mathcal F(J)} V\left(R_{F} \left(-h_I(N^{F})\right) \right),\\
&\geq -C\sharp \mathcal F(J)+\int_{\mathbb S^d}h_I \,d\sigma_J.
\end{align*}
For an integral convex exhaustion $(\mathsf J_n)_n$ we obtain finally for large $n$ by using Lemma \ref{pepe} and Lemma \ref{anglee}   
\begin{align*}\sharp \underline{\partial^-\mathsf J_n}&\geq  -C(I,\frac{\angle \mathsf J_\infty}{2}) \cdot \sharp \mathcal F(\mathsf J_n)+\int_{\mathbb S^d}h_I \,d\sigma_{\mathsf J_n},\\
\liminf_n\frac{\sharp \underline{\partial ^{-} \mathsf J_n} }{p(\mathsf J_n)}&\geq \lim_n\int_{\mathbb S^d}h_I \,d\sigma_{\tilde{\mathsf J_n}},\\
&\geq \int_{\mathbb S^d}h_I \,d\sigma_{\mathsf J_\infty}.
\end{align*}

One proves similarly that $\limsup_n \frac{\sharp \underline{\partial^{+} \mathsf J_n} }{p(\mathsf J_n)}\leq \int_{\mathbb S^d}h_I \,d\sigma_{\mathsf J_\infty}$ and this concludes the 
proof  of Proposition \ref{fdff} as we have  $\sharp \underline{\partial ^{+} \mathsf J_n}\geq \sharp \underline{\partial ^{-} \mathsf J_n}$ according to Lemma \ref{prof}.
\end{proof}

\subsection{Supremum of  $O\mapsto V_I(O)$}\label{gros}
 In this section we investigate the supremum of  $V_I$ on $\mathcal D^1$ for a given convex polytope $I$ of $\mathbb R^d$. We recall that there is a unique sphere $S_I$ containing $I$ with minimal radius, usually called the \textbf{smallest bounding sphere} of $I$.  
We let  $R_I$ and $x_I$ be respectively the radius  and the center of $S_I$. There are at least two distinct points in $S_I\cap I$, whenever $I$ is not reduced to a singleton, and $S_I\cap I\subset \ex(I)$. Moreover we have the following alternative : 
\begin{itemize}
\item either there is a finite subset of  $S_I\cap I$ generating an inscribable  polytope $T$  with $\Int(T)\ni x_I$   (in particular the interior set of $I$ is non empty), 
\item  or there is a hyperplane $H$ containing $x_I$ such that $I$ lies in  an associated  semispace and $S_I\cap H$ 
is the smallest bounding sphere of $I\cap H$.
\end{itemize}
The smallest bounding sphere $S_I$ (or $I$ itself) will be said \textbf{nondegenerated} (resp. \textbf{degenerated}) and  an associated polytope $T$ (resp. hyperplane $H$) is said  \textbf{generating}. 
For an inscribable polytope $T$ in $\mathbb R^d$ we may define its dual $T'$ as  the polytope given by the intersection of the inner semispaces tangent to the  circumsphere of $T$ at the vertices of $T$.  In the following $T'$ always denotes the dual polytope of a generating polytope $T$ with respect to $I$.

When $S_I$ is degenerated, there is a sequence of affine spaces $H=H_1\supset H_2\supset \cdots H_l\ni x_I$ such that $I\cap H_l$ is nondegenerated in $H_l$ and for all $1\leq i< l$ the convex polytope $I\cap H_i$ is degenerated  in $H_i$ with $H_{i+1}$ as an associated generating hyperplane ($H_i$ is a $d-i$ dimensional affine space).  We denote by  $L$ a generating polytope of $I\cap H_l$ in $H_l$ and by $L'$ its dual  polytope in $H_l$. Let $U$ be  an isometry  of $\mathbb R^d$  mapping $H_i$ for $i=1,\cdots, l$ to $\{0_i\}\times \mathbb{R}^{d-i}$ (where $0_i$ denotes the origin of $\mathbb R^i$) with $U(x_I)=0$. Then for $R>0$ we let 
$T'_R:=U^{-1}\left([-R, R]^l\times U(L')\right)$. The faces $F$ of $T'_R$ satisfy \begin{enumerate}
\item either $F=U^{-1}\left( [-R, R]^l\times U(\mathsf F)\right)$ for some face $\mathsf F$ of $L'$, 
\item  or $F=U^{-1}\left([-R, R]^{l-1}\times \{\pm R\}_i\times U(L')\right)$ for $i=1,\cdots, l$ (where $\{\pm R\}_i$ coresponds to the $i^{th}$ coordinate of the product).
\end{enumerate}
For $i=1,2$ we let   $\mathcal F_i (T'_R)$  be the subset of $\mathcal F(T'_R)$ given by the faces of the $i^{\text{th}}$ category. 

 Observe that when $x_I$ coincide with the origin then $T'$ or $T'_R$, $R>0$ are convex domains.

\begin{prop}\label{supe}
$$\sup_{O\in \mathcal D^1}V_I(O)=R_I.$$
The supremum of $V_I$  is achieved  if and only if  $S_I$ is nondegenerated. The supremum is then achieved for $O\in \mathcal D^1$  homothetic to the  dual  polytope $T'$ of a generating polytope $T$. 
\end{prop}

\begin{proof}
For any $v\in \mathbb R^d$ we have
\begin{align*}
V_{I+v}(O)&=\int h_{I+v} \,d\sigma_O,\\
&=\int h_{I} \,d\sigma_O+\int_{\mathbb S^d}v\cdot u\, d\sigma_O(u),\\
& =\int h_{I} \,d\sigma_O+ \int_{\partial O}v\cdot N^O\, d\mathcal H_{d-1}.
\end{align*}

By the divergence formula we have $\int_{\partial O}v\cdot N^O\, d\mathcal H_{d-1}=0$ for any $v\in \mathbb R^d$ and $O\in \mathcal D^1$.  Therefore we may assume $x_I=0$.  With the above notations we have $\max_{i\in I}i\cdot v\leq R_I$ for all
$ v\in \mathbb R^d$  with $\|v\|=1$ 
with equality iff $v$ belongs to $R_I^{-1}I$. Therefore $V_I(O)\leq R_I$ for any $O\in \mathcal D^1$. Moreover if the equality occurs then for $x$ in a subset $E$ of $\partial O$ with full $\mathcal H_{d-1}$-measure,  $h_I\left(N^O(x)\right)=\max_{i\in I}i\cdot x=R_I$ and therefore  the normal unit vector $N^O(x)$ belongs to $R_I^{-1}I$. But as $O$ is a convex domain, we may find $d+1$ points $x_1,\cdots, x_{d+1}$  in $E$ in such a way the origin belongs to the interior of the simplex $T=R_I\cv\left(N^O(x_1),\cdots , N^O(x_{d+1})\right)$. Thus $S_I$ is nondegenerated and the polytope  $T$ is a generating polytope with respect to $I$.  Moreover we have   with the above notations 
 $$\int h_{I} \,d\sigma_{T'}=R_I p(T').$$
Therefore the homothetic polytope $O'$ of $T'$ with unit perimeter achieves the supremum of $V_I$.
We  consider  now the degenerated case. With the above notations, we have   $h_I(N^F)=R_I$ for any $F\in \mathcal F_1(T'_R)$ (recall we assume $x_I=0$ without loss of generality). Moreover 
$\mathcal H_{d-1}\left(\bigcup_{F\in \mathcal F _2 (T'_R)}F \right)=o(p(T'_R))$ when $R$ goes to infinity. Therefore  the renormalization $O_R\in \mathcal D^1$ of $T'_R$ satisfies $$V_I(O_R)\xrightarrow{R\rightarrow +\infty} R_I.$$



\end{proof}

\section{Cellular automata}\label{deff}

\subsection{Definitions}
 We consider a finite set $\mathcal A$.  We endow the set $\mathcal A$ with the discrete topology and  $X_d=\mathcal A^{\mathbb Z^d}$  with the product topology.   We consider the $\mathbb Z^d$-shift $\sigma$ on $\mathcal A^{\mathbb Z^d}$ defined for $l\in \mathbb Z^d$ and $u=(u_k)_k\in X_d$ by $\sigma^l(u)=(u_{k+l})_k$. Any closed subset $X$  of  $X_d$ invariant under the action of $\sigma$ is called a $\mathbb Z^d$-\textbf{subshift}. We fix such a subshift  $X$ in  the  remaining of the paper.

For a bounded subset $J$ of $\mathbb{R}^d$ we consider the   partition $\mathsf P_J$ into \textbf{$\underline{J}$-cylinders}, i.e.  the element $\mathsf P_J^x$ of $\mathsf P_J$ containing $x=(x_i)_{i\in \mathbb{Z}^d}\in X$ is given by $\mathsf P_J^x:=\{y=(y_i)_{i\in \mathbb Z^d}\in X, \ \forall i\in \underline{J} \ y_i=x_i \}$. In other terms we may define $\mathsf P_J$ as the joined partition $\bigvee_{j\in \underline{J}}\sigma^{-j}\mathsf P_0$ with $\mathsf P_0$ being the zero-coordinate partition. 

A \textbf{cellular automaton} (CA for short) defined on a $\mathbb Z^d$-subshift $X$ is a continuous map $f:X\rightarrow X$ which commutes with the shift action $\sigma$. By a famous theorem of Hedlund \cite{He} the cellular automaton $f$ is given by a local rule, i.e. there exists a finite subset $I$ of $\mathbb{Z}^d$ and a map $F:\mathcal  A^I\rightarrow \mathcal A$ such that 
$$\forall j\in \mathbb{Z}^d \ \left( fx \right)_j=F\left( \left(x_{j+i}\right)_{i\in I}\right).$$ The (smallest) subset $I$ is called the \textbf{domain} of the CA.  Recall $I'=I\cup \{0\}$ and let $\mathbb I$ be the convex hull of $I'$.


\subsection{Lyapunov exponents for higher dimensional cellular automata}\label{Lypa}

Lyapunov exponent of  one-dimensional cellular automata have been defined in \cite{Sh,Ti}. We develop a similar theory in higher dimensions. Let $f$ be a CA on a $\mathbb Z^d$-subshift $X$ with domain  $I$.  

Given a convex  body $J$ of $\mathbb R^d$  and  $x\in X$, we let  
$$\mathcal E_f(x,J):= \{K \text{ convex body}, \  f \mathsf P_J^x \subset \mathsf P_{K}^{fx} \}$$

A priori the family $\mathcal E_f(x,J)$ does not admit   a greatest  element for the inclusion. Observe also that  the convex body $J\ominus I$ belongs to  $\mathcal E_f(x,J)$, in particular this family is not empty. Then we let for all $x$ :
$${\gr}_Jf(x):=\min\{\sharp \underline{ J\setminus K}, \ K\in \mathcal E_f(x,J)\} .$$

The family $\mathcal E_f(x,J)$ and the function ${\gr}_Jf(x)$ are constant on each atom $A$ of $\mathsf P_J$, thus we let $\mathcal E_f(A,J)$ and  ${\gr}_Jf(A)$ be these quantities. We denote by   $\mathcal D_f(x,J)$ the subfamily of   $\mathcal E_f(x,J)$ consisting in $K$ with $\sharp  \underline{J\setminus K}={\gr}_Jf(x)$. For $K$ in $\mathcal D_f(x,J)$ the intersection $K\cap J$ defines a convex body, which belongs also to $\mathcal D_f(x,J)$.

 For a convex exhaustion $\mathcal J=(J_n)_n$, we define \textbf{the growth  ${\gr}_{\mathcal J}f$} with respect to $\mathcal J$  as the  following real functions on $X$ :
$${\gr}_{\mathcal J}f:=\limsup_n\frac{ {\gr}_{J_n}f }{p(J_n)}.$$
Finally we let for  a convex domain $O\in \mathcal D^1$ : $${\gr}_{O}f=\sup_{\mathcal J\in \mathcal{E}(O)}{\gr}_{\mathcal J}f.$$

\begin{lemma}
The sequence of functions  $\left({\gr}_{O} f^k\right)_k$ is subadditive, i.e. 
$$\forall k,l\in \mathbb N \ \forall x\in X, \ {\gr}_{O}f^{k+l}(x)\leq {\gr}_{O}f^l(f^kx)+{\gr}_{O}f^k(x).$$
 \end{lemma}

\begin{proof}

Fix $x\in X$ and $k,l\in \mathbb N$.  Let $\mathcal J=(J_n)_n \in\mathcal E(O)$.  We consider a sequence $\mathcal K:=(K_n)_n$ of convex bodies in $\prod_n \mathcal D_{f^k}(x,J_n)$ with $K_n\subset J_n$ for all $n$.  Let $I_k$ be the domain of $f^k$. The convex body $J_n\ominus I_k$ belongs to $\mathcal E_{f^k}(x,J_n)$ for all $n$. By Proposition \ref{zut}, we have $\sharp \underline{ J_n\setminus K_n} \leq \sharp \underline{\partial_{I_k}^-J_n} =O\left(p(J_n)\right)$.
It follows from Lemma \ref{nul} and Remark \ref{integr} that  $\mathcal K$ is a convex exhaustion in $\mathcal E(O)$  with $p(K_n)\sim^n p(J_n)$. We also let $\mathcal L=(L_n)_n\in  \prod_n \mathcal D_{f^l}(f^kx,K_n)$ with $L_n\subset K_n$ for all $n$. Similarly the sequence $\mathcal L$ belongs to $\mathcal E(O)$  with $p(L_n)\sim^n p(J_n)$. Then we have for all positive integers $n$ : 
 \begin{align*}
 f^{k+l} \mathsf P_{J_n}^x &=f^l(f^k \mathsf P_{J_n}^x),\\
 & \subset f^l\left(\mathsf P_{K_n}^{f^kx}\right),\\
 &\subset \mathsf P_{L_n}^{f^{k+l}x}.
 \end{align*} 
Therefore we have 
\begin{align*}
{\gr}_{J_n}f^{k+l}(x)&\leq \sharp \underline{J_n\setminus L_n},\\
&\leq  \sharp \underline{J_n\setminus K_n}+ \sharp \underline{K_n\setminus L_n},
\end{align*}
then 
\begin{align*}
{\gr}_{\mathcal J}f^{k+l}(x)&=\limsup_n\frac{ {\gr}_{J_n}f }{p(J_n)},\\
&\leq \limsup_n\frac{ {\gr}_{K_n}f }{p(J_n)}+\limsup_n\frac{ {\gr}_{L_n}f }{p(J_n)},\\
&\leq \limsup_n\frac{ {\gr}_{K_n}f }{p(K_n)}+\limsup_n\frac{ {\gr}_{L_n}f }{p(L_n)}, \\
& \leq {\gr}^+_{\mathcal K}f^k(x)+{\gr}^+_{\mathcal L}f^l(f^kx),\\
\end{align*}

As the sequence $\mathcal K$ and $\mathcal L$ lie in $\mathcal E(O)$ we conclude that

\begin{align*}
{\gr}_{O}f^{k+l}(x) & \leq {\gr}_{O}f^k(x)+{\gr}_{O}f^l(f^kx).
\end{align*}

\end{proof}

The nonnegative function ${\gr}_{O}f$ satisfies ${\gr}_{O}f\leq \sup_{\mathcal J\in \mathcal E(O)}\limsup_n\frac{\sharp \underline{\partial_I ^- J_n}}{p\left(J_n\right)} $ and this last term is finite according to Proposition \ref{fin}. Therefore the subadditive ergodic theorem applies : for any $\mu\in \mathcal M(X,f)$
the sequence $\left(\frac{1}{n} {\gr}_{O}f^{n}(x)\right)_k$ converge almost everywhere to a $f$-invariant function $\chi_O$ with $\int \chi_O\, d\mu=\lim/\inf_n\frac{1}{n}\int {\gr}_{O}f^{n}\, d\mu$. We call the function $\chi_O$ the Lyapunov exponent of $f$ with respect to $O$.

\begin{rem} The exponent  $\chi_O$ for $O\in \mathcal D$ plays some how the role of the sum of the positive Lyapunov exponents in smooth dynamical systems. 
\end{rem}

\section{Rescaled entropy of cellular automata}\label{entr}

\subsection{Definition}
We let $\mathcal M(f)$ (resp. $\mathcal M(f,\sigma)$) be the set of invariant Borel probability measures on $X$ which are $f$-invariant (resp. $f$- and $\sigma$-invariant). 
For a finite clopen partition $\mathsf P$ of $X$ we let $H_{top}(\mathsf P)=\log \sharp \mathsf P$ and $H_{\mu}(\mathsf P)=-\sum_{A\in \mathsf P}\mu(A)\log\mu(A)$ with $\mu\in \mathcal M(f)$. In the following the symbol $*$ denotes either $*=top$ or $*=\mu\in \mathcal M(f)$. We let $h_{*}(f,\mathsf P)$ be the  entropy with respect to the clopen partition $\mathsf P$ :
$$h_{*}(f,\mathsf P):=\lim_n\frac{1}{n}H_*\left(\bigvee_{k=0}^{n-1}f^{-k}\mathsf P\right).$$
For two partitions $\mathsf P$, $\mathsf Q$  of $X$,    we say $\mathsf P$ is finer than $\mathsf Q$ and we write $\mathsf P >\mathsf Q$, when any atom of $\mathsf P$ is contained in an atom of $\mathsf Q$. The functions $H_*(\cdot)$ and $h_{*}(f,\cdot )$ are nondecreasing with respect to this order.

The rescaled  entropy with respect to a convex  exhaustion $\mathcal J=(J_n)_n$ is defined as follows 
$$h^d_{*}(f,\mathcal J)=\limsup_{n}\frac{h_{*}(f,\mathsf P_{J_n})}{p(J_n)}.$$
   In \cite{Lak} the authors defines a similar notion for the rescaled topological entropy with the renormalization factor  $\sharp \partial^-_{I}\underline{J_n}$ (which depends on the domain $I$ of  $f$) rather than $p(J_n)$. 

\begin{rem}
For $d=2$, when $J=\bigcup_{i\in I}J_i$ is a finite disjoint union of  Jordan domains $J_i$ with Lipshitz boundary, we have 
\begin{align*}\frac{h_{top}(f, \mathsf P_J)}{p(J)}&\leq \frac{\sum_{i\in I}h_{top}(f,\mathsf P_{J_i})}{\sum_{i\in I}p(J_i)},\\
&\leq \sup_{i\in I}\frac{h_{top}(f,\mathsf P_{J_i})}{p(J_i)}.
\end{align*} 
Moreover for each $i$, we have $p(J_i)\geq p\left(\cv(J_i)\right)$ and $\mathsf P_{\cv(J_i)}$ is finer than $\mathsf P_{J_i}$. Therefore 
\begin{align*}\frac{h_{top}(f, \mathsf P_J)}{p(J)}&\leq \frac{\sum_{i\in I}h_{top}(f,\mathsf P_{J_i})}{\sum_{i\in I}p(J_i)},\\
&\leq \sup_{i\in I}\frac{h_{top}(f,\mathsf P_{\cv(J_i)})}{p\left(\cv(J_i)\right) }.
\end{align*}
This inequality justifies somehow that we focus on convex bodies $J$ of $\mathbb R^d$. 
\end{rem}

We let  also for any $O\in \mathcal D^1$ $$h^d_{*}(f,O)=\sup_{\mathcal J \in \mathcal E(O)}h^d_{*}(f, \mathcal J)$$ and 
$$h^d_{*}(f)=\sup_{\mathcal J}h^d_{*}(f, \mathcal J),$$
where the last supremum holds over all convex exhaustions $\mathcal J$.
  For $d=1$ we have $p(J)=2$ for any convex subset $J$. Therefore up to a factor $2$ we recover the usual definition of entropy, $2h^1_{*}(f)=h_{*}(f)$. 

\begin{rem}\label{encore}
As the CA  $f$ commutes with the shift action $\sigma$ we have for all $k\in \mathbb Z^d$ and any subset $J$ of $\mathbb Z^d$
$h_{top}(f, \mathsf P_{J+k})= h_{top}(f, \sigma^{-k}\mathsf P_{J})=h_{top}(f,\mathsf  P_J)$ and the same holds for the measure theoretical entropy with respect to measures in $\mathcal M(f,\sigma)$. Let us call generalized convex domain any convex body with a non empty interior set.  Replacing convex domains by generalized convex domains, we may define generalized convex exhaustions  $\mathcal J$ and the associated  rescaled entropies. Then it follows from the aforementioned invariance by translation of the entropy, that  $h^d_{top}(O)=h^d_{top}(O+\alpha)$ for all $\alpha\in \mathbb{R}^d$ and all generalized convex domain $O$ with unit perimeter. Indeed for any $(J_n)_n\in \mathcal E(O)$ (resp. $\mathcal E  \left(\mathcal{O+\alpha}\right)$) there is a sequence of integers $(k_n)_n$ with $(J_n+k_n)_n\in \mathcal E  \left(\mathcal{O+\alpha}\right)$ (resp. $(J_n)_n\in \mathcal E(O)$). 
\end{rem}

\begin{rem}
\begin{enumerate}
\item The partition $\mathsf P_{J_n}$ may be written as $\bigvee_{k\in J_n}\sigma^{-k}\mathsf P_{0}$ with $\mathsf P_0$ being the zero-coordinate partition. Instead of $\mathsf P_0$  we could choose another clopen generating partition $\mathsf P$, i.e. a partition of $X$ into clopen sets with $\bigvee_{k\in \mathbb Z^d}\sigma^{-k}\mathsf P$ equal to the partition of $X$ into points. But for a finite subset $J$ of $\mathbb{Z}^d$ we have $\bigvee_{k\in J}\sigma^{-k}\mathsf P>\mathsf P_0$ and $\bigvee_{k\in J}\sigma^{-k}\mathsf P_0>\mathsf P$ so that in the definition of the rescaled entropy we may replace $\mathsf P_0$ by any other generator $\mathsf P$ of $X$, i.e. $\mathsf P_{J_n}$ by $\bigvee_{k\in \underline{J_n}} \sigma^{-k}\mathsf P$.    
\item Let $X$ be a zerodimensional compact metrizable space endowed with a  expansive $\mathbb Z^d$-action $\tau$.  We consider a map $f$ preserving $(X,\tau)$ i.e. $f$ is an homeomorphism of $X$ commuting with $\tau$. The triple $(X,\tau, f)$ is called a topological $\mathbb Z^d$-expansive preserving system (t.e.p.s. for short). Two 
t.e.p.s. $(Y,\phi, g)$ are conjugated when there is a homeomorphism $h:X\rightarrow Y$ such  that $h\circ f\circ h^{-1}=g$  and $h\circ \tau \circ h^{-1}=\phi$. 
 We may define the rescaled entropy as we did for a CA and all the previous results hold in this more general setting.  Moreover two conjugated t.e.p.s. have the same rescaled entropy. Any t.e.p.s. is conjugated to a CA. 
 \end{enumerate}
 \end{rem}

 \subsection{Link with the metric mean dimension}
 In a compact metric space $(X,\mathsf d)$, the ball of radius $\epsilon\geq 0$ centered at $x\in X$ will be denoted by $B_{\mathsf d}(x,\epsilon)$. 
For a continuous map $f:X\rightarrow X$  we denote by $\mathsf d_n$ the dynamical distance  defined for all $n\in \mathbb N$ by 
$$\forall x,y\in X,\ \mathsf d_n(x, y)=\max \{\mathsf d(f^kx,f^ky), \ 0\leq k<n\}.$$ 
The metric mean dimension of $f$ is defined 
as $\mdim(f,\mathsf d)=\limsup_{\epsilon\rightarrow 0}\frac{h_{top}(f,\epsilon)}{|\log \epsilon|}$ where $h_{top}(f,\epsilon)$ denotes the topological entropy at the scale $\epsilon>0$ : 
$$h_{top}(f,\epsilon):=\limsup_n\frac{1}{n}\log \min \{\sharp C, \ \bigcup_{x\in C}B_{\mathsf d_n}(x,\epsilon)=X\}.$$

The topologial mean dimension is the infimum of $\mdim(f,\mathsf d)$ over all distances on $X$. We refer to \cite{Lin} for alternative definitions and furter properties of mean dimension.  The topological mean dimension of a finite dimensional topological system is null.

 Here $f$ is a CA on a subshift of $\mathbb Z^d$. In particular it has zero topological mean dimension.  For a norm $\|\cdot\|$ of $\mathbb R^d$ we may associate a metric $\mathsf d_{\|\|}$ on $X_d$ by letting  $\mathsf d_{\|\|}(u,v)=\alpha^{-\min\{\|k\|, \ k\in \mathbb{Z}^d, \ u_k\neq v_k\}}$ for all $u=(u_k)_k, v=(v_k)_k\in X_d$.  Then for $l\in \mathbb N$ the (open) ball $B_{\mathsf d_{\|\|}}(x,2^{-l})$ with respect to $\mathsf d_{\|\|}$ coincides with the cylinder  $\mathsf P_{J_l}^x$ with  $J_l=B_{\|\|}(0,l)$. 

As there is a correspondence between convex symmetric domains  and unit balls of norms  on $\mathbb R^d$, the mean dimension with respect to such distances $\mathsf d_{\|\|}$   are given by $h_{top}^d(f,\mathcal J_O)$ for convex symmetric domains $O$.

 \begin{rem}
 In \cite{Tsu} the authors work with a measure theoretical quantity, called the measure distorsion rate dimension and show a variational principle with the metric mean dimension of $d_{\|\|}$. Does this quantity coincides with $\mu\mapsto h^d_\mu (f, O)$ with $O$ being the symmetric convex  domain associated to the norm $\|\|$? 
 \end{rem}

\subsection{Monotonicity and Power} 
We investigate now basic properties of the rescaled entropy.

\begin{lemma}\label{nwe}
For any $O\in \mathcal D$ and any $\alpha>0$, we have 
$$h^d_*(f,\mathcal J_O)=h^d_*(f,\mathcal J_{\alpha O}).$$ 

\end{lemma}

\begin{proof}
For $n\in \mathbb N$, we let $k_n=\lceil \frac{n}{\alpha}\rceil$, thus $nO\subset k_n\alpha O$ and $p(nO)\sim^n p(k_n\alpha O)$. 
Therefore 
\begin{align*}
h^d_*(f,\mathcal J_O)&=\limsup_n \frac{h_*(f,\mathsf P_{nO} )}{ p(nO)},\\
&\leq \limsup_n \frac{h_*(f,\mathsf P_{k_n\alpha O} )}{ p(nO)},\\
&\leq \limsup_n \frac{h_*(f,\mathsf P_{k_n\alpha O} )}{ p(k_n\alpha O)},\\
&\leq h^d_*(f,\mathcal J_{\alpha O}).
\end{align*} 
The other inequality is obtained by considering $\alpha O$ and $\alpha^{-1}$ in place  of $O$ and $\alpha$.  
\end{proof}

\begin{lemma}\label{dfd}For any $O\in \mathcal D^1$ and  $O'\in \mathcal D$ with $O\subset \Int(O')$, we have 
$$h^d_*(f,\mathcal J_O)\leq h^d_*(f,O)\leq p(O') h^d_*(f,\mathcal J_{O'}).$$
\end{lemma}

\begin{proof}
As $\mathcal J_O\in \mathcal E(O)$  the inequality $h^d_*(f,\mathcal J_O)\leq h^d_*(f,O)$  follows from the definitions. 
Let now $\mathcal J\in \mathcal E(O)$. For $n$ large enough we have $\tilde J_n\subset \Int(O')$, therefore $J_n\subset p(J_n)^{\frac{1}{d-1}}O'$.
Therefore we conlude that  \begin{align*}
h^d_*(f,\mathcal J)&\leq \limsup_n \frac{p\left(p(J_n)^{\frac{1}{d-1}}O'\right)}{p(J_n)}h^d_*(f, \mathcal J_{O'}),\\
& \leq p(O')h^d_*(f, \mathcal J_{O'}).
\end{align*}
\end{proof}

For $O\in \mathcal D^1$ the origin belongs to $\Int(O)$ so that  $\alpha O\in \mathcal D$ and  $O\subset \Int(\alpha O)$ for any $\alpha>1$. Moreover we have $h^d_*(f,\mathcal J_{\alpha O})=h^d_*(f,\mathcal J_{O})$ by Lemma \ref{nwe}.  Together with Lemma \ref{dfd} we get immediately :
\begin{coro}
$$\forall O\in \mathcal D^1, \ h^d_*(f,O)=h^d_{*}(f, \mathcal J_O).$$
\end{coro}

\begin{coro}

\[O\mapsto h_*^d(f,O)\text{ is continuous on }\mathcal D^1.\]

\end{coro}

 Convex  polytopes are dense in $\mathcal D$. Therefore  we get  with $\mathcal{P}$ being the collections of convex $d$-polytopes with the origin in their interior set :
 
 \begin{coro}\label{poly}
 $$\sup_{O\in \mathcal D^1}h^d_*(f,O)= \sup_{P\in \mathcal{P}}h^d_*(f,\mathcal J_P) .$$
 \end{coro}
However we will see that the supremum is not always achieved.  We prove now a formula for the rescaled entropy of a power. 
    
    \begin{lemma}\label{zdz}
    $$\forall O\in \mathcal D^1 \ \forall k\in \mathbb N,\  h^d_*(f^k,O)=kh^d_*(f,O).$$
    \end{lemma}
    
\begin{proof}Let $O\in \mathcal D^1$ and  $\mathcal J=(J_n)_n\in \mathcal E(O)$. Let $J_n^k=J_n\oplus\underbrace{I\oplus\cdots\oplus I}_{k \text{ times}}$ for all $n$. The sequence $\mathcal J^k=(J_n^k)_n$  belongs also to $\mathcal E(O)$. Moreover the partition $\mathsf P _{J_n^k}$ is finer than $\bigvee_{l=0}^{k-1}f^{-l}\mathsf P_{J_n}$. Therefore
\[ h_*(f^k, \mathsf P_{J_n}) \leq k h_*(f,\mathsf P_{  J_n}) = h_*\left(f^k, \bigvee_{l=0}^{k-1}f^{-l}\mathsf P_{J_n}\right)\leq  h_*(f^k,\mathsf P _{J_n^k})\]
and we then obtain 
\[ h_*^d(f^k,\mathcal J) \leq  k h_*^d(f,\mathcal J)\leq h_*^d(f^k, \mathcal J^k).\] We conclude by taking the supremum in $\mathcal J\in \mathcal E (O)$. 
\end{proof}

  \begin{rem}
Clearly we have $h^d_\mu\leq h^d_{top}$ for any $\mu\in \mathcal M(f)$ but we ignore if a general variational principle holds true. 
\end{rem}

\subsection{A first upperbound for the rescaled entropy}  
  
Let $(X,f)$ be a cellular automaton with domain $I$. We relate the entropy of $\mathsf P_J$ with the entropy of $\mathsf P_{\partial^\pm J}$ and we prove an upperbound for  the rescaled entropy $h_{top}^d(f,O)$ in term of the first relative quermass integral $V_{\mathbb{I}}(O)$ with $\mathbb I$ being the convex hull of $I'$.
  
  \begin{lemma}\label{entropie} For any bounded subset $J$ of $\mathbb R^d$, we have 
$$ h_*(f,\mathsf P_J)=h_*(f, \mathsf P_{\partial_I^- J}) \text{ and }  h_*(f,\mathsf P_J)\leq h_*(f, \mathsf P_{\partial_I^+ J}). $$
  \end{lemma}
  
\begin{proof}The inequality $ h_*(f,\mathsf P_J)\geq h_*(f, \mathsf P_{\partial_I^- J})$  follows directly from the inclusion $\partial^-J\subset J$. By definition of the domain $I$ and the erosion $J\ominus I$, we have $P_J>f^{-1}P_{J\ominus I}$. Therefore   we get $f^{-1}\mathsf P_J\vee \mathsf P_J=f^{-1}\mathsf P_{\partial^- J}\vee P_J$ and then by induction $\mathsf P_{J}\vee \bigvee_{l=0}^{k-1}f^{-l}\mathsf P_{\partial^- J}= \bigvee_{l=0}^{k-1}f^{-l}\mathsf P_{J}$ for all $k$. We conclude that : \begin{align*} h_*(f,\mathsf P_J)&=\lim_k\frac{1}{k}H_*(f,\bigvee_{l=0}^{k-1}f^{-l}\mathsf P_{J}), \\
&\leq \lim_k\frac{1}{k}\left(H_*\left(\mathsf P_{J}\right)+H_*\left(\bigvee_{l=0}^{k-1}f^{-l}\mathsf P_{\partial^- J} \right)\right),\\
&\leq h_*(f, \mathsf P_{\partial^- J}).
\end{align*}

We also have   \[\mathsf P_J\vee \mathsf P_{\partial^+J}>\mathsf P_{J\oplus I}>f^{-1}\mathsf P_J.\] Therefore we get now by induction on $k$
 \[\mathsf P_J\vee \bigvee_{l=0}^{k-2}f^{-l}\mathsf P_{\partial^+ J}>\bigvee_{l=0}^{k-1}f^{-l}\mathsf P_{J}.\]  
 This implies $ h_*(f, \mathsf P_{\partial_I^+ J})\leq h_*(f,\mathsf P_J)$.

\end{proof}

\begin{prop}\label{rutop}
For any $O\in \mathcal D^1$, 
$$h^d_{top}(f,O)\leq V_{\mathbb I}(O)\log \sharp \mathcal A.$$
\end{prop}  
  
  \begin{proof}
Recall that   
\begin{align*}
h^d_{top}(f,O)&=h^d_{top}(f,\mathcal J_O),\\
&=\limsup_n\frac{h_{top}(f,\mathsf P_{nO})}{p(nO)}.
\end{align*}
Then by applying Lemma \ref{entropie} we obtain 
\begin{align*}
h^d_{top}(f,O)&\leq \limsup_n\frac{h_{top}(f,\mathsf P_{\partial^{\pm}nO})}{p(nO)}, \\
&\leq \limsup_n\frac{ \sharp \underline{\partial^{\pm}nO} \log \sharp \mathcal A}{p(nO)}. 
\end{align*}
For all $k\in \mathbb N\setminus\{0\}$ we let $I_k$ be the domain of $f^k$ and we denote by $\mathbb I_k$ the convex hull of $I'_k=I_k\cup\{0\}$.
Clearly we have $I_k \subset \underbrace{I\oplus \cdots \oplus I}_{k \text{ times }}$,  therefore $\mathbb I_k\subset k\mathbb I$.
By Lemma \ref{grg}, we get for some constant $C=C(d)$ : 
\begin{align*}
h^d_{top}(f^k,O)&\leq \left(V_{\mathbb I_k}(O)+C\right)\log \sharp \mathcal A,\\
&\leq \left(V_{k\mathbb{I}}(O)+C\right)\log \sharp \mathcal A,\\
&\leq \left(kV_{\mathbb{I}}(O)+C\right)\log \sharp \mathcal A.
\end{align*}

But by Lemma \ref{pwr} we have $h^d_{top}(f^k,O)=kh^d_{top}(f,O)$, so that we finally conclude  when $k$ goes to infinity 
$$h^d_{top}(f,O)\leq V_{\mathbb{I}}(O)\log \sharp \mathcal A.$$
  \end{proof}

\section{Ruelle inequality}\label{Rue}
Recall $(X,\sigma)$ denotes a $\mathbb Z^d$-subshift. The topological entropy of $\sigma$ is  defined for any F\"olner sequence $\mathcal L=(L_n)_n$  (see e.g. \cite{Wei})  as  \[h_{top}(\sigma)=\limsup_n\frac{H_{top}(\mathsf P_{L_n})}{\sharp L_n}.\]

\begin{lemma}\label{cov}
 For all $\epsilon>0$ there exists  $c>0$ such that we have for any $K\subset J$  convex bodies:
$$H_{top}(\mathsf P_{ J\setminus K})\leq \left(\sharp \underline{J\setminus K}+cp(J\oplus \mathsf C)\right)\cdot (h_{top}(\sigma)+\epsilon).$$
\end{lemma}

\begin{proof}

Let $\epsilon>0$. As the sequence of cubes $\mathcal C= (C_n)_n$ defined by $C_n=[-n,n[^d\cap\mathbb Z^d$ is a F\"olner sequence, there is
a positive integer $m$ such that $\frac{H_{top}\left (\mathsf P_{C_m}\right)}{\sharp C_m} < h_{top}(\sigma)+\epsilon$.  Then for some $c=c(m)>0$ we may cover $\underline{J\setminus K}$ by a family  $\mathcal F$ at most
$\frac{\sharp \underline{J\setminus K} +cp(J\oplus \mathsf C)}{\sharp C_m}$ disjoint translated copies of $C_m$. Indeed if $\mathsf R_m$ denotes a partition of $\mathbb{R}^d$ into translated copies of $C_m$, then any atom $A$ of $\mathsf R_m$ with $\underline{A}\cap \left(J\setminus K\right)\neq \emptyset$ either satisfies $\underline{A}\subset  J\setminus K$ or $\underline{A}\cap \left(\partial^-_{C_m}J\cup \partial^-_{C_m}K\right)\neq \emptyset$. Clearly the number of $A$'s in the first case is less than $\frac{\sharp \underline{J\setminus K}}{\sharp C_m}$, whereas  the numbers of atoms $A$ satisfying the second condition is less than $\sharp 
\underline{\partial^-_{C_m}J}+\sharp \underline{
\partial^-_{C_m}K}$. Arguing as in the proof of Proposition \ref{fin}, this 
last term is less than $c\left(p(J\oplus \mathsf C)+p(K\oplus\mathsf C)\right)$ for some constant 
$c$ depending on $m$. As $K$ is contained in $J$ we have $p(J\oplus 
\mathsf C)\leq p(K\oplus \mathsf C)$.

Therefore 
\begin{align*}
H_{top}(\mathsf P_{ J\setminus K})&\leq\left(\sharp \underline{J\setminus K}+2cp(J\oplus \mathsf C)\right)\frac{H_{top}\left(\mathsf P_{C_m}\right)}{\sharp C_m},\\
&\leq\left(\sharp \underline{J\setminus K}+2cp(J\oplus \mathsf C)\right)\cdot  (h_{top}(\sigma)+\epsilon).
\end{align*}
\end{proof}

We refine now the inequality obtained in Lemma \ref{rutop} at the level of invariant measures : 
\begin{lemma}
$$\forall \mu\in \mathcal M(f), \ h_\mu(f, O)\leq h_{top}(\sigma)\int \chi_O\,d\mu.$$
 \end{lemma}
 
 \begin{proof}
 For any convex domain $J$ and any $\mu\in \mathcal M (f)$ we have 
\begin{align*}
h_\mu(f, \mathsf P_J)& \leq H_\mu(f^{-1}\mathsf P_J| \mathsf P_J),\\
& \leq \sum_{A\in \mathsf P_J}\mu(A)H_{\mu_A}(f^{-1}\mathsf P_J).
\end{align*}
 Fix $\epsilon>0$ and let $c$ be as in Lemma \ref{cov}.  Then if $(K_A)_{A\in \mathsf P_J}$ is a family of convex bodies in $\prod_{A\in \mathsf P_J}\mathcal E_f(A, J)$ with $K_A \subset J$ for all $A$ we obtain
\begin{align*}
h_\mu(f, \mathsf P_J)& \leq \sum_{A\in \mathsf P_J}\mu(A)H_{\mu_A}(f^{-1}\mathsf P_{J \setminus K_A}),\\
&\leq \sum_{A\in \mathsf P_J}\mu(A) H_{top}(\mathsf P_{J\setminus K_A}),\\
&\leq \sum_{A\in \mathsf P_J}\mu(A) \left(\sharp \underline{J\setminus K_A}+cp(J\oplus \mathsf C)\right)\cdot (h_{top}(\sigma)+\epsilon).
\end{align*}
By choosing $K_A$ with $\sharp \underline{ J\setminus K_A}$ minimal we obtain 
\begin{align*}
h_\mu(f, \mathsf P_J)&\leq   \left(h_{top}(\sigma)+\epsilon\right)\cdot \left( \int {\gr}_Jf \, d\mu +cp(J\oplus \mathsf C) \right).
\end{align*}

Therefore we have for any convex exhaustion $\mathcal J=(J_n)_n$ (recall that $p(J_n\oplus \mathsf C)\sim^n p(J_n)$) :  
\begin{align*}
h^d_\mu(f,\mathcal J)&=\limsup_n\frac{h_\mu(f, \mathsf P_J)}{p(J_n)},\\
&\leq   \left(h_{top}(\sigma)+\epsilon\right)\cdot \left( \limsup_n\int  \frac{{\gr}_{J_n}f}{p(J_n)}\, d\mu +c \right).
\end{align*}
By Proposition \ref{fin} we have for all $x\in X$
\[\sup_{n\in \mathbb{N}} \frac{{\gr}_{J_n}f(x)}{p(J_n)}\leq \sup_{n\in \mathbb{N}} \frac{\sharp \underline{\partial^-J_n}}{p(J_n)}<+\infty.\]
We may therefore apply Fatou's Lemma to the sequence of functions $\left(-\frac{{\gr}_{J_n}f}{p(J_n)}\right)_n$ :
\[\limsup_n\int  \frac{{\gr}_{J_n}f}{p(J_n)}\, d\mu\leq \int \limsup_n \frac{{\gr}_{J_n}f}{p(J_n)}\, d\mu,\]
then 
\begin{align*}
h^d_\mu(f,\mathcal J) &\leq \left( h_{top}(\sigma)+\epsilon\right)\left(\int {\gr}_{\mathcal J}f  \, d\mu+c\right).
\end{align*}
By taking the supremum over $\mathcal J\in  \mathcal E(O)$ we get 
\begin{align*}
h^d_\mu(f,O)& \leq \left( h_{top}(\sigma)+\epsilon\right)\left( \int {\gr}_{O}f \, d\mu+c\right).
\end{align*}
By Lemma \ref{zdz} we have $\frac{h^d_\mu(f^k, O) }{k}=h^d_\mu(f, O)$ for any $k$. 
Apply the above inequality to $f^k$ :
\begin{align*}
h^d_\mu(f,O)&\leq  \left( h_{top}(\sigma)+\epsilon\right)\left( \int  \frac{{\gr}_{O}f^k}{k}  \, d\mu+\frac{c}{k}\right). 
\end{align*}
When $k$ goes to infinity and then $\epsilon$ goes to zero, we conclude $h^d_\mu(f,O)\leq h_{top}(\sigma)\int \chi_O\, d\mu$.

 \end{proof}

\section{Entropy formula for permutative CA}\label{las} 
The cellular automaton $f$ is said \textbf{permutative} at $i\in \mathbb Z^d$ if for all pattern $P$ on $I\setminus \{i\}$  and for all $a\in \mathcal A$ there is $b\in \mathcal A$ such that the pattern $P^i_b$ on $I\cup\{i\}$ given by the completion of $P$ at $i$ by $b$ satisfies $F(P^i_b)=a$, in particular $i$ belongs to the domain $I$ of $f$.   The CA is said permutative when it is permutative at the nonzero extreme points of the convex hull $\mathbb I$  of $I'=I\cup \{0\}$ (these points lie in $I$). The algebraic CA as described in the introduction are permutative.  

\begin{prop}\label{form}
The topological rescaled entropy of a permutative  CA $f$ on $X_d$ is given by 
$$h^{d}_{top}(f)=R_{I'}\log \sharp \mathcal A.$$ 
\end{prop}

The sets $I'$ and $\mathbb I$ have the same smallest bounding sphere, thus $R_{I'}=R_{\mathbb I}$. Theorem \ref{pm},  stated in the introduction,  follows from Proposition \ref{form}.

\begin{ques}
For a permutative CA, the uniform measure $\lambda^{\mathbb{Z}^d}$ with $\lambda$ being the uniform measure on $\mathcal A$ is known to be invariant \cite{Wi}.
Does the uniform measure maximize the entropy ?
\end{ques}

Recall that for any $k\in \mathbb N\setminus\{0\}$ we denote by $I_k$ the domain of $f^k$ and $\mathbb I_k$ the convex hull of $I'_k=I_k\cup \{0\}$. In the following we also let $C(P,L)=\{(x_i)_{i\in \mathbb Z^d}\in X, \ x_j=p_j \, \forall j\in L \}$ be the cylinder associated to the pattern $P=(p_j)_{j\in L}\in \mathcal A^L$ on  $L\subset \mathbb Z^d$. We also write $C(P) $ for this cylinder when there is no confusion on $L$. 

\begin{lemma}\label{pwr}For any permutative CA $f$ and any $k\in \mathbb N\setminus\{0\}$,  the CA $f^k$ is also permutative and 
$$\mathbb I_k=k\mathbb I.$$
\end{lemma}

\begin{proof}As already observed, the inclusion  $\mathbb I_k\subset k\mathbb I$ holds for any CA (not necessarily permutative). We will show  $k\ex(\mathbb I)\subset  I'_k$, which implies together with $\mathbb I_k\subset k\mathbb I$ the equality $\mathbb I_k= k\mathbb I$. Let $i\in \ex(\mathbb I)\setminus \{0\}\subset I$. For a fixed $k$ we prove by induction on $k$ that $f^k$ is permutative at $ki$, in particular $ki\in I'_k$. Let $P$ be a pattern on $I_k\setminus \{ki\}$ and let $a\in \mathcal A$. Since we have $I_k\subset I_{k-1}\oplus I$, we may complete $P$ by a pattern $Q$ on $ \left(I_{k-1}\oplus I\right)\setminus \{ki\}$. By induction hypothesis, $(k-1)i$ lies in $ \ex(\mathbb I_{k-1})$ and $i$ lies in $\ex(\mathbb I)$, 
 therefore $ki$ does not belong to $I_{k-1}\oplus \left( I\setminus\{i\}\right)$, so that we have $I_{k-1}\oplus \left( I\setminus\{i\}\right)\subset \left(I_{k-1}\oplus I\right)\setminus \{ki\}$. Therefore there  is a pattern $R$ on $I\setminus \{i\}$ such that   $f^{k-1}C\left(Q,  \left(I_{k-1}\oplus I\right)\setminus \{ki\} \right)$ is contained in the cylinder $C(R,I\setminus\{i\})$. As $f$ is permutative at $i$ there is $b\in \mathcal A$ with $F(R_b^i)=a$ or in other terms $f\left(C(R_b^i,I)\right)\subset C\left(a,\{0\}\right)$. Since $f^{k-1}$ is permutative at $(k-1)i$, we may find $c\in \mathcal A$ with $f^{k-1}\left(C(Q^{ki}_c,I_{k-1}\oplus I)\right)\subset C\left(b, \{i\}\right)$. Therefore we get 
 \[f^k\left(C(Q^{ki}_c,I_{k-1}\oplus I)\right)\subset f\left( C(R^i_b,I)\right)\subset   C\left(a, \{0\}\right).\] But $I_k$ is the domain of $f^k$ and $P$ is the restriction of $Q$ to $I_k\setminus \{ki\}$, so that we also have 
 $f^k\left(C(P^{ki}_c,I_{k})\right)\subset C\left(a, \{0\}\right)$, i.e. $f^k$ is permutative at $ki$. 
\end{proof}

For a convex $d$-polytope $J$ and a face $F$ of $J$ we consider the subset of $\partial_{\mathbb I}^-J$ given by   $\partial^-_{\mathbb I}F:=\partial_{\mathbb I}^-J\cap T_F^+J(-h_{\mathbb I}(N^F))$. 
The sets $\partial^-_{\mathbb I}F$ for $F\in \mathcal F(J)$ are covering $\partial^-_{\mathbb I}J$ but do not define a partition in general. 
For any $F\in \mathcal F(J)$ we let $u^F\in \ex(\mathbb I)\subset I'$ with $u^F\cdot N^F=h_{\mathbb I}(N^F)$ and we also let $d_F$ be the the Euclidean distance to $T_F$. 
Then for $j\in\underline{\partial^-_{\mathbb I}J}$ we let $F_j$ be a face of $J$ such that  $d_{F_j}(j+u^{F_j})=-d_{F_j}(j)+u^{F_j}\cdot N^{F_j}$ is maximal among faces $F$ with $j\in \partial^-_{\mathbb I}{F}$. We consider then a total order $\prec$ on $\underline{\partial^-_{\mathbb I}J}$ such that $i\prec j$ if $d_{F_{i}}(i+u^{F_{i}})<d_{F_j}(j+u^{F_j})$.  We  also let $\mathcal F_{\mathbb I}(J)$ be the subset of $\mathcal F(J)$ given by faces $F$ for which $u_F$ is uniquely defined.  We denote by $\partial_{\mathbb I}^{\bot}J$ the subset of $\partial^-_{\mathbb I}J$ given by  $$\partial_{\mathbb I}^{\bot}J:=\bigcup_{F\in \mathcal F_{\mathbb I}(J)}\partial^-_{\mathbb I}F.$$

\begin{lemma}\label{encore}With the above notations, let $j\in \underline{\partial_{\mathbb I}^{\bot}J}$. Then $$\forall k\in \mathbb N, \ j+ku^{F_j} \notin \{j',\, j'\prec j\}\oplus k\mathbb I.$$ 
\end{lemma}

\begin{proof}
We argue by contradiction : there are $j'\prec j$ and $u\in \mathbb  I $ with $j+ku^{F_j}=j'+ku$. Observe that 
\begin{align*}
d_{F_j}( j+ku^{F_j} )&=d_{F_j}(j+u^{F_j})+(k-1)u^{F_j}\cdot N^{F_j},\\
d_{F_j}( j'+ku )&= d_{F_j}(j'+u)+(k-1)u\cdot N^{F_j}.
\end{align*}
We will show that the equality between these two distances implies $u=u^{F_j}$, therefore $j=j'$.  Indeed we have 
\begin{align*}
 d_{F_j}( j'+u ) &\leq \sup_{v\in \ex(\mathbb I)}  d_{F_j}( j'+v ),    & u\cdot N^{F_j} &\leq  \sup_{v\in \ex(\mathbb I)} v\cdot N^{F_j},\\
                      &\leq   d_{F_{j'}}( j'+u^{F_{j'}} ),                                                         &                        &\leq h_{\mathbb I}(N^{F_j}),\\
 d_{F_j}( j'+u ) &\leq d_{F_j}(j+u^{F_j})                                                                  &        u\cdot N^{F_j}   & \leq u^{F_j}\cdot N^{F_j},
\end{align*}
therefore  $u\cdot N^{F_j}=u^{F_j}\cdot N^{F_j}$, and finally $u=u^{F_j}$ as $j$ belongs to $\underline{\partial_{\mathbb I}^{\bot}J}$.
\end{proof}

For a partition $\mathsf P$ of $X$ and a positive integer $k$, we write  $\mathsf P^k$ to denote the iterated partition $\bigvee_{l=0}^{k-1}f^{-l}\mathsf{P}$ in order to simplify the notations.
\begin{lemma}\label{zer}
Let $J$ be a convex $d$-polytope and let $k,n$ be positive integers.  For any $A^k\in \mathsf P_{J}^k$ and any pattern $P$ on $\underline{\partial^{\bot}_{\mathbb I} J}$, there is  $w\in A^k$ such that  $f^kw$ belongs to $C(P, \underline{\partial^{\bot}_{\mathbb I} J})$.
\end{lemma}

\begin{proof}
For any $j\in  \partial^{\bot}_{\mathbb I} J$ we let $P_j$ be the restriction of $P=(p_l)_{l\in \partial^\bot J}$ to $\{j',\, j'\prec j\}$. We show now by induction on $j\in \underline{\partial^\bot J}$ that there is $w\in A^k$ with $f^kw\in 
C(P_j)$.  By Lemma \ref{pwr} the CA $f^k$ is permutative at $ku^{F_j}$ so that we may change the $(j+ku^{F_j})^{\text{th}}$-coordinate of $w$ to get $w'\in X$ with $(f^kw')_j=p_j$. Moreover  the $j'$-coordinates of $f^kw$ for $j'\prec j$ only depends on the coordinates of $w$ on $\{j',\, j'\prec j\}\oplus k\mathbb I$ so that by Lemma \ref{encore} we still have $f^kw'\in C(P_j, \{j',\, j'\prec j\})$, thus  $f^kw'\in C(P_{j''})$ with $j''$ being the successor of $j$ for  $\prec$ in $\underline{\partial^\bot J}$. 
\end{proof}

\begin{lemma}\label{lat}
Let $T'$ and $T'_R$, $R>0$  be the polytopes associated to $\mathbb I$ as defined in  Subsection \ref{gros}. We have 
$$\mathcal F(T')=\mathcal F_{\mathbb I}(T')$$
and $$\forall R>0, \ \mathcal F_1(T'_R)\subset \mathcal F_{\mathbb I}(T'_R).$$
\end{lemma}
\begin{proof}
Let $F\in \mathcal F(T')$ or $F\in \mathcal F_1(T'_R)$. Such a face $F$ is tangent to $S_{I'}$ at some $u\in \ex(\mathbb I)$ with $u\cdot N^F=h_\mathbb{I}(N^F)$. Then any $v$ with $v\cdot N^F=h_\mathbb{I}(N^F)$ belongs to $T_F$. But $T_F\cap \mathbb I \subset T_F\cap S_{I'}=\{u\}$, therefore we have necessarily $u_F=u$.

\end{proof}

We are now in a position to prove Proposition \ref{form}. 
\begin{proof}[Proof of Proposition \ref{form}]

 The inequality $h^{d}_{top}(f)\leq R_{I'}\log \sharp \mathcal A$ follows immediately from Proposition \ref{rutop} and Proposition \ref{supe}. 
 By Lemma \ref{zer} we have for any convex $d$-polytope $O$ and any positive integer $n$
 $$\forall A^k\in\mathsf P_{nO}^{k}, \  \sharp\{ A^{k+1}\in \mathsf P_{nO}^{k+1}, \ A^{k+1}\subset A^k\}\geq \sharp\underline{\partial^\bot nO}.$$ 
 Consequently we have 
\begin{align*}
h_{top}(f, \mathsf P_{nO})&\geq \sharp\underline{\partial^\bot nO}\log \sharp \mathcal A,\\
h_{top}^d(f, \mathcal{J}_{O})&\geq  \limsup_n\frac{\sharp\underline{\partial^\bot nO}}{n^{d-1}p(O)}  \log \sharp \mathcal A.
\end{align*}

We first assume that $S_{\mathbb I}=S_{I'}$  is nondegenerated. Let $T'$ be the dual polytope of a generating polytope $T$. Note that $T'$ is a convex body with nonempty interior containing $0$ (but the origin does not lie necessarily in its interior set).  By Lemma \ref{lat}
we have $\mathcal F(T')=\mathcal F_{\mathbb I}(T')$, therefore $\mathcal F(nT')=\mathcal F_{\mathbb I}(nT')$ and $\partial^\bot nT'=\partial^- nT'$ for all $n$.  
Applying then Lemma \ref{grg} we get for some constant $C=C(d)$ :
\begin{align*}
h_{top}^d(f, \mathcal{J}_{T'})&\geq  \limsup_n\frac{\sharp\underline{\partial^- nT'}}{n^{d-1}p(T')}  \log \sharp \mathcal A,\\
&\geq \frac{V_{\mathbb{I}}(T')}{p(T')}  \log \sharp \mathcal A-C.
\end{align*}

Then  it follows from Proposition  \ref{supe} that :   
\begin{align*}
h_{top}^d(f, \mathcal{J}_{T'})& \geq  R_{\mathbb I}  \log \sharp \mathcal A-C.
\end{align*}

 For any positive integer $k$ the above equality also holds for $f^k$ and $\mathbb{I}_k$ in place of $f$ and $\mathbb I$. 

Moreover we have $\mathbb I_k=k\mathbb I$ according to Lemma \ref{pwr}, so that we get together with  the power formula of Lemma \ref{zdz} and $O':=p(T')^{-\frac{1}{d-1}}T'$ :
\begin{align*}
h_{top}^d(f, O')&   =\frac{h_{top}^d(f^k, O')}{k},\\
&\geq \frac{R_{\mathbb{I}_k}}{k} \log \sharp \mathcal A-\frac{C}{k},\\
&\geq \frac{R_{k\mathbb{I}}}{k} \log \sharp \mathcal A-\frac{C}{k},\\
&\geq  R_{\mathbb I}\log \sharp \mathcal A-\frac{C}{k},\\
h_{top}^d(f, T')&\geq R_{I'} \log \sharp \mathcal A.
\end{align*}
This conclude the proof in  the nondegenerated case. 

We deal now with the degenerated case. By Lemma \ref{lat} we have for all $R>0$ with the notations of Subsection \ref{gros} :
\[h_{top}^d(f, \mathcal{J}_{T'_R})\geq  \limsup_n\frac{\sharp\underline{\partial^- nT'_R}-\sum_{F\in \mathcal F_2(T'_R) }\sharp \underline{\partial^- nF } }{p(nT'_R)}  \log \sharp \mathcal A.\]
But for $F\in \mathcal{F}_2(T'_R)$ we have 
\begin{align*}
\sharp \underline{\partial^- nF }&\leq V(\partial^- nF \oplus \mathsf C),\\
&=n^{d-1}\diam(\mathbb I) O(R^{l-1})
\end{align*}
Since $\lim_{R\rightarrow \infty}\frac{p(T'_R)}{R^l}=\mathcal{H}_{d-l}(L')>0$ and $\sharp \mathcal F_2(T'_R)= 2l$, we get 
\[ \limsup_n\frac{\sum_{F\in \mathcal F_2(T'_R) }\sharp \underline{\partial^- nF } }{p(nT'_R)}=\diam(\mathbb I) O(R^{-1}).\]
Together with Proposition \ref{grg} we get for some constant $C=C(d)$ :
\[h_{top}^d(f, \mathcal{J}_{T'_R})\geq \left( V_{\mathbb{I}}(T'_R)-C-\diam(\mathbb I) O(R^{-1})\right)  \log \sharp \mathcal A.\]


We  conclude as in the degenerated case by using the power rule. Fix $\epsilon>0$ and let $k>C\epsilon^{-1}$. We obtain finally 
\begin{align*}
h_{top}^d(f, O'_R)&=\frac{h_{top}^d(f^k, O'_R)}{k},\\
& \geq \left(\frac{V_{\mathbb{I}_k}(T'_R)}{kp(T'_R)}-\epsilon -\frac{\diam(\mathbb I_k)}{k}O(R^{-1})\right)  \log \sharp \mathcal A,\\
& \geq \left(\frac{V_{\mathbb{I}}(T'_R)}{p(T'_R)}-\epsilon -\diam(\mathbb I)O(R^{-1}) \right)  \log \sharp \mathcal A,\\
&\xrightarrow{R\rightarrow +\infty} (R_{I'}-\epsilon)\log \sharp \mathcal A. 
\end{align*}

\end{proof}

 \end{document}